\documentclass[12pt]{article}
\usepackage{amsfonts,amssymb,amsmath,amsthm}
\usepackage[russian,english]{babel}
\date{}

\theoremstyle{plain}
\newtheorem{theorem}{Theorem}
\newtheorem{proposition}{Proposition}
\newtheorem{lemma}{Lemma}
\newtheorem{corollary}{Corollary}
\theoremstyle{definition}
\newtheorem{definition}{Definition}
\numberwithin{equation}{section}
\numberwithin{theorem}{section}
\numberwithin{proposition}{section}
\numberwithin{lemma}{section}
\numberwithin{corollary}{section}
\numberwithin{definition}{section}

\newcommand{\R}{\mathbb{R}}
\newcommand{\N}{\mathbb{N}}
\newcommand{\Z}{\mathbb{Z}}
\newcommand{\D}{\mathcal{D}}
\newcommand{\T}{\mathbb{T}}
\newcommand{\esslim}{\operatornamewithlimits{ess\,lim}}
\newcommand{\esslimsup}{\operatornamewithlimits{ess\,limsup}}
\newcommand{\essinf}{\operatornamewithlimits{ess\,inf}}
\newcommand{\esssup}{\operatornamewithlimits{ess\,sup}}
\newcommand{\sign}{\operatorname{sign}}
\newcommand{\Co}{\operatorname{co}}
\newcommand{\<}{\langle}
\renewcommand{\>}{\rangle}
\newcommand{\MV}{\operatorname{MV}}
\newcommand{\const}{\mathrm{const}}

\newcommand{\supp}{\operatorname{supp}}
\title{On long time behavior of periodic entropy solutions of a degenerate non-linear parabolic equation}
\author{Evgeny Yu. Panov}
\begin{document}
\maketitle
\begin{abstract}
We prove the asymptotic convergence of a space-periodic entropy solution of a one-dimensional degenerate parabolic
equation to a traveling wave. It is also shown that on a segment containing the essential range of the limit profile the flux function is linear (with the slope equaled to the speed of the traveling wave) and the diffusion function is constant.
\end{abstract}
\maketitle

\section{Introduction}\label{sec1}
In the half-plane $\Pi=\R_+\times\R$, where $\R_+=(0,+\infty)$, we consider the nonlinear parabolic equation
\begin{equation}
\label{1}
u_t+\varphi(u)_x-g(u)_{xx}=0,
\end{equation}
where the functions $\varphi(u),g(u)\in C(\R)$, and $g(u)$ is non-strictly increasing. Since $g(u)$
may be constant on nontrivial intervals, (\ref{1}) is a degenerate parabolic equation.  In particular,
for $g(u)\equiv\const$ this equation reduces to the first-order conservation law
\begin{equation}
\label{cl}
u_t+\varphi(u)_x=0.
\end{equation}
We study the Cauchy problem for equation (\ref{1}) with initial condition
\begin{equation}\label{2}
u(0,x)=u_0(x)\in L^\infty(\R).
\end{equation}
We recall the notion of entropy solution of (\ref{1}), (\ref{2}) (see \cite{Car,MalT,AndMal}).

\begin{definition}\label{def1}
A bounded measurable function $u=u(t,x)\in L^\infty(\Pi)$ is called an entropy solution (e.s. for short) of (\ref{1}), (\ref{2}) if the generalized derivative $g(u)_x\in L^2_{loc}(\Pi)$ and for each $k\in\R$
\begin{equation}
\label{entr}
|u-k|_t+[\sign(u-k)(\varphi(u)-\varphi(k))]_x-|g(u)-g(k)|_{xx}\le 0
\end{equation}
in the sense of distributions on $\Pi$ (in $\D'(\Pi)$), and
\begin{equation}\label{3}
\esslim_{t\to 0} u(t,\cdot)=u_0 \ \mbox{ in } L^1_{loc}(\R).
\end{equation}
\end{definition}

Condition (\ref{entr}) means that for any non-negative test function $f=f(t,x)\in C_0^\infty(\Pi)$
\begin{equation}
\label{entr1}
\int_{\Pi}\{|u-k|f_t+[\sign(u-k)(\varphi(u)-\varphi(k))]f_x+|g(u)-g(k)|f_{xx}\}dtdx\ge 0.
\end{equation}
Taking in (\ref{entr}) $k=\pm M$, where $M\ge\|u\|_\infty$, we obtain that an e.s. satisfies equation (\ref{1}) in $\D'(\Pi)$, i.e. it is a weak solution of this equation.

Remark also that in the case of equation (\ref{cl}) condition (\ref{entr}) coincides with the known Kruzhkov entropy
condition \cite{Kr}. In the case of one space variable there always exists the unique e.s. of (\ref{1}), (\ref{2}), see
\cite{MalT}.

Now we suppose that the initial function $u_0(x)$ is periodic: $u_0(x+1)=u_0(x)$ a.e. in $\R$. Such a function may be considered as a function on a circle $\T=\R/\Z$
(we can identify $\T$ with a fundamental interval $[0,1)$~): $u_0(x)\in L^\infty(\T)$. Since $u(t,x+1)$ is an e.s. of the same problem (\ref{1}), (\ref{2}), then, by the uniqueness of e.s., $u(t,x+1)=u(t,x)$ a.e. in $\Pi$. This means that
$u(t,x)$ is a space periodic function, i.e. $u(t,x)\in L^\infty(\R_+\times\T)$. The main our result is the following asymptotic property.

\begin{theorem}\label{thM}
There is a periodic function $v(y)\in L^\infty(\T)$ (a profile) and a constant $c\in\R$ (a speed) such that
\begin{equation}\label{ass}
\esslim_{t\to+\infty} (u(t,x)-v(x-ct))=0 \ \mbox{ in } L^1(\T).
\end{equation}
Moreover, $\displaystyle\int_\T v(y)dy=I\doteq\int_\T u_0(x)dx=\int_0^1 u_0(x)dx$ and the functions $\varphi(u)-cu$, $g(u)$ are constant on
the minimal segment $[\alpha(v),\beta(v)]$ containing values $v(y)$ for almost all $y\in\T$ (i.e., $\alpha(v)=\essinf v(y)$, $\beta(v)=\esssup v(y)$).
\end{theorem}

In the case of conservation laws (\ref{cl}) Theorem~\ref{thM} was established in \cite{PaMZM}.

Taking into account that in the case $\alpha(v)<\beta(v)$ the interval $(\alpha(v),\beta(v))$ contains $I$, we derive the following decay property.

\begin{corollary}\label{cor1}
Assume that for any $c\in\R$ the functions $\varphi(u)-cu$, $g(u)$ are not constant simultaneously in any vicinity of $I$. Then $v(y)\equiv I$, that is,
\begin{equation}\label{dec}
\esslim_{t\to+\infty} u(t,x)=I \ \mbox{ in } L^1(\T).
\end{equation}
\end{corollary}

\section{Preliminaries}
We denote $z^+=\max(0,z)$, $\sign^+ z=(\sign z)^+=\left\{\begin{array}{lr} 1, & z>0, \\ 0, & z\le 0.\end{array}\right.$

\begin{lemma}\label{lem1}
Let $u_1=u_1(t,x)$, $u_2=u_2(t,x)$ be e.s. of (\ref{1}), (\ref{2}) with periodic initial data $u_{01}(x),u_{02}(x)$. Then for a.e. $t,s>0$, $t>s$
\begin{eqnarray}\label{contr+}
\int_\T (u_1(t,x)-u_2(t,x))^+dx\le\int_\T (u_1(s,x)-u_2(s,x))^+dx\le \nonumber\\ \int_\T (u_{01}(x)-u_{02}(x))^+dx, \\
\label{cons}
\int_\T (u_1(t,x)-u_2(t,x))dx=\int_\T (u_{01}(x)-u_{02}(x))dx,
\end{eqnarray}
\end{lemma}

\begin{proof}
By the Kato inequality obtained in \cite{Car}
\begin{equation}\label{kato}
[(u_1-u_2)^+]_t+[\sign^+(u_1-u_2))(\varphi(u_1)-\varphi(u_2))]_x-[(g(u_1)-g(u_2))^+]_{xx}\le 0
\end{equation}
in $\D'(\Pi)$. Let $\alpha(t)\in C_0^\infty(\R_+)$, $\beta(y)\in C_0^\infty(\R)$ be nonnegative functions and $\int_\R \beta(y)dy=1$.
Applying (\ref{kato}) to the test function $\alpha(t)\beta(x/r)$, where $r>0$, we arrive at the inequality
\begin{eqnarray}\label{kato1}
\int_\Pi (u_1(t,x)-u_2(t,x))^+\alpha'(t)\beta(x/r)dtdx+ \nonumber\\
r^{-1}\int_{\Pi}\sign^+(u_1-u_2)(\varphi(u_1)-\varphi(u_2))\alpha(t)\beta'(x/r)dtdx+ \nonumber\\
r^{-2}\int_{\Pi}(g(u_1)-g(u_2))^+\alpha(t)\beta''(x/r)dtdx\ge 0.
\end{eqnarray}
As is rather well-known (see, for example, \cite[Lemma~2.1]{PaDU1}), for a bounded spatially periodic function $w(t,x)\in L^\infty(\Pi)$ the following relation holds
\begin{equation}\label{lm}
\lim_{r\to\infty}r^{-1}\int_{\Pi}w(t,x)a(t)b(x/r)dtdx=
C\int_0^{+\infty}\left(\int_\T w(t,x)dx\right) a(t)dt,
\end{equation}
where $a(t)\in C_0(\R_+)$, $b(y)\in C_0(\R)$, $\displaystyle C=\int_{\R} b(y)dy$.
Multiplying (\ref{kato1}) by $r^{-1}$ and passing to the limit as $r\to\infty$ with the help of (\ref{lm}), we arrive at
$$
\int_0^{+\infty} I(t)\alpha'(t)dt\ge 0
$$
for any $\alpha(t)\in C_0^\infty(\R_+)$, $\alpha(t)\ge 0$, where $\displaystyle I(t)=\int_\T(u_1(t,x)-u_2(t,x))^+dx$. This means that $I'(t)\le 0$ in $\D'(\R_+)$ and therefore,
for almost all $t,s$, $t>s>0$
$$
I(t)\le I(s)\le \esslim_{t\to 0} I(t)=I(0)\doteq\int_\T (u_{01}(x)-u_{02}(x))^+dx,
$$
and (\ref{contr+}) follows.

We use also that, by initial condition (\ref{3}),
$$
|I(t)-I(0)|\le\int_0^1 |u_1(t,x)-u_{01}(x)|dx+\int_0^1 |u_2(t,x)-u_{02}(x)|dx\to 0,
$$
as time $t\to 0$, running over a set of full Lebesgue measure.

To establish (\ref{cons}) we observe that
$$
(u_1-u_2)_t+(\varphi(u_1)-\varphi(u_2))_x-(g(u_1)-g(u_2))_{xx}=0 \ \mbox{ in } \D'(\Pi)
$$
since $u_1$, $u_2$ are weak solutions of (\ref{1}).
Applying this relation to the test function $r^{-1}\alpha(t)\beta(x/r)$ and passing to the limit as $r\to\infty$, we obtain the identity
$$
\int_0^\infty\left(\int_\T (u_1(t,x)-u_2(t,x))dx\right)\alpha'(t)dt=0
$$
$\forall \alpha(t)\in C_0^\infty(\R_+)$. This evidently implies (\ref{cons}).
The proof is complete.
\end{proof}

It readily follows from (\ref{contr+}) that $u_1(t,x)\le u_2(t,x)$ a.e. in $\Pi$ whenever $u_{01}\le u_{02}$ (comparison principle). Clearly, this implies the uniqueness of periodic e.s. Another direct consequence of Lemma~\ref{lem1} is the following contraction property in $L^1(\T)$: for a.e. $t,s>0$, $t>s$
\begin{equation}\label{contr}
\int_\T |u_1(t,x)-u_2(t,x)|dx\le\int_\T |u_1(s,x)-u_2(s,x)|dx\le\int_\T |u_{01}(x)-u_{02}(x)|dx.
\end{equation}
Indeed, (\ref{contr}) readily follows from (\ref{contr+}) and the identity
$|u_1-u_2|=(u_1-u_2)^++(u_2-u_1)^+$.

Taking in (\ref{cons}) $u_1=u$, $u_2=0$, we derive the following mass conservation
property
\begin{equation}\label{mass}
\int_\T u(t,x)dx=I=\int_\T u_0(x)dx.
\end{equation}

\medskip
We will need the notion of a \textit{measure-valued function}. Recall (see \cite{Di,Ta}) that a
measure-valued function on $\Pi$ is a weakly measurable map $(t,x)\mapsto \nu_{t,x}$ of $\Pi$ into the space $\operatorname{Prob}_0(\R)$ of probability Borel measures with compact support in~$\R$.

The weak measurability of $\nu_{t,x}$ means that for each continuous function $g(\lambda)$ the
function $\displaystyle (t,x)\to\<\nu_{t,x}(\lambda),g(\lambda)\>=\int g(\lambda)d\nu_{t,x}(\lambda)$ is measurable on~$\Pi$.

We say that a measure-valued function $\nu_{t,x}$ is {\it bounded\/} if there exists $R>0$ such that
$\supp\nu_{t,x}\subset [-R,R]$ for almost all $(t,x)\in\Pi$. We shall denote by
$\|\nu_{t,x}\|_\infty$ the smallest such~$R$.

Finally, we say that measure-valued functions of the kind
$\nu_{t,x}(\lambda)=\delta(\lambda-u(t,x))$, where $u(t,x)\in L^\infty(\Pi)$ and
$\delta(\lambda-u^*)$ is the Dirac measure at $u^*\in\R$, are {\it regular}. We identify these
measure-valued functions and the corresponding functions $u(t,x)$, so that there is a natural
embedding $L^\infty(\Pi)\subset\MV(\Pi)$, where $\MV(\Pi)$ is the set of bounded measure-valued
functions on~$\Pi$.

Measure-valued functions naturally arise as weak limits of bounded sequences in $L^\infty(\Pi)$ in
the sense of the following theorem by Tartar (see~\cite{Ta}).

\begin{theorem}\label{thT}
Let $u_m(t,x)\in L^\infty(\Pi)$, $m\in\N$, be a bounded sequence. Then there exist a subsequence
$u_n(t,x)$ and a measure-valued function $\nu_{t,x}\in\MV(\Pi)$ such that\begin{equation} \label{4}
\forall p(\lambda)\in C(\R) \quad p(u_n) \mathrel{\mathop{\rightharpoonup}_{n\to\infty}}
\<\nu_{t,x}(\lambda),p(\lambda)\> \quad\text{weakly-\/$*$ in } L^\infty(\Pi).
\end{equation}
Besides, $\nu_{t,x}$ is regular, i.e., $\nu_{t,x}(\lambda)=\delta(\lambda-u(t,x))$ if and only if
$u_n(t,x) \mathrel{\mathop{\to}\limits_{n\to\infty}} u(t,x)$ in $L^1_{loc}(\Pi)$.
\end{theorem}

Assume that $u_n$ is a bounded in $L^\infty(\Pi)$ sequence of e.s. to approximate
equations
\begin{equation}\label{1a}
u_t+\varphi_n(u)_x-(g_n(u))_{xx}=0,
\end{equation}
where $\varphi_n(u),g_n(u)\in C(\R)$, $g_n(u)$ are nondecreasing functions, and  the sequences $\varphi_n(u)\to\varphi(u)$, $g_n(u)\to g(u)$ as $n\to\infty$
uniformly on any segment. Suppose that the sequence $u_n$ converges as $n\to\infty$ to a measure valued function $\nu_{t,x}\in\MV(\Pi)$ in the sense of relation (\ref{4}). The following property was established in \cite[Theorem 3.5]{PaSIMA1}
by a new variant of compensated compactness theory developed in \cite{PaAIHP}.

\begin{theorem}\label{th1}
For a.e. $(t,x)\in\Pi$ the function $g(\lambda)$ is constant and the function $\varphi(\lambda)$
is affine on the convex hull $\Co\supp\nu_{t,x}$ of the closed support $\supp\nu_{t,x}$.
\end{theorem}

\section{Main results}
Let $u(t,x)\in L^\infty(\R_+\times\T)$ be a bounded space-periodic function such that $\displaystyle\int_\T u(t,x)dx=I=\const$, and $\gamma_n$, $n\in\N$, be a positive sequence. We consider the sequence $u_n=u(\gamma_nt,nx)$ bounded in $L^\infty(\R_+\times\T)$.

\begin{lemma}\label{lem2}
The sequence $u_n\rightharpoonup I$ as $n\to\infty$ weakly-$*$ in $L^\infty(\Pi)$.
\end{lemma}

\begin{proof}
We consider the complex-valued trigonometric functions $e^{2\pi i kx}$, $k\in\Z$, $i^2=-1$. By translation invariance of the measure $dx$ on the circle $\T$
\begin{eqnarray*}
\int_\T u(\gamma_nt,nx)e^{2\pi ikx}dx=e^{-2\pi ik/n}\int_T u(\gamma_nt,nx+1)e^{2\pi ik(x+1/n)}dx=\\ e^{-2\pi ik/n}\int_T u(\gamma_nt,nx)e^{2\pi ik x}dx.
\end{eqnarray*}
This implies that for $n>|k|>0$
$$
\int_\T u(\gamma_nt,nx)e^{2\pi ikx}dx=0=I\int_\T e^{2\pi ikx}dx
$$
while for $k=0$
$$
\int_\T u(\gamma_nt,nx)dx=\int_\T u(\gamma_nt,y)dy=I=I\int_\T dx.
$$
This equalities imply that
$$
\lim_{n\to\infty} \int_{\R_+\times\T}u_n(t,x)f(t,x)dtdx=\int_{\R_+\times\T} If(t,x)dtdx
$$
for test functions $f(t,x)=a(t)e^{2\pi ikx}$, where $a(t)\in L^1(\R_+)$, $k\in\Z$.
Since linear combinations of such functions are dense in $L^1(\R_+\times\T)$,
we conclude that $u_n\rightharpoonup I$ as $n\to\infty$ weakly-$*$ in $L^\infty(\R_+\times\T)$ and, in view of the periodicity, also weakly-$*$ in $L^\infty(\Pi)$. The proof is complete.
\end{proof}

Now, let $u(t,x)$ be an e.s. of (\ref{1}), (\ref{2}) with a periodic initial function $u_0(x)\in L^\infty(\T)$, and $M=\|u\|_\infty$, $\displaystyle I=\int_\T u_0(x)dx$.

\begin{proposition}\label{pro1}
Assume that the function $\varphi(u)$ is not affine in any vicinity of $I$. Then
the decay relation (\ref{dec}) holds.
\end{proposition}

\begin{proof}
We consider the sequence $u_n=u(nt,nx)$. As is easy to see, $u_n$ is an e.s. of the equation
$$
u_t-\varphi(u)_x-\frac{1}{n}g(u)_{xx}=0
$$
for each $n\in\N$. Notice that $g_n(u)=\frac{1}{n}g(u)\to 0$ uniformly on any segment. Extracting a subsequence if necessary, we can assume that the sequence $u_n$ converges as $n\to\infty$ to a measure valued limit function $\nu_{t,x}\in\MV(\Pi)$ in the sense of (\ref{4}). In view of (\ref{mass}) the sequence $u_n$ satisfies the assumptions of Lemma~\ref{lem2}. By this lemma
$\displaystyle u_n\rightharpoonup I=\int_\T u_0(x)dx$ as $n\to\infty$ weakly-$*$ in $L^\infty(\Pi)$. From (\ref{4}) it follows that for a.e. $(t,x)\in\Pi$
\begin{equation}\label{5}
I=\int\lambda d\nu_{t,x}(\lambda).
\end{equation}
By Theorem~\ref{th1} the flux function $\varphi(u)$ is affine on
$[a(t,x),b(t,x)]=\Co\supp\nu_{t,x}$ for a.e. $(t,x)\in\Pi$. If $a(t,x)<b(t,x)$ then it follows from (\ref{5}) that $(a(t,x),b(t,x))$ is a neighborhood of $I$ and by our assumption $\varphi(u)$ is not affine on $[a(t,x),b(t,x)]$. Therefore, $a(t,x)=b(t,x)=I$ for a.e. $(t,x)\in\Pi$. This means that $\nu_{t,x}(\lambda)=\delta(\lambda-I)$. By Theorem~\ref{thT}
$u_n\to I$ in $L^1_{loc}(\Pi)$. Passing to a subsequence $u_{n_k}$, we can suppose that $u_{n_k}=u(n_kt,n_kx)\to I$ in $L^1(\T)$ for a.e. $t>0$. It is possible to choose such $t=t_0$ with additional property $n_kt_0\in E$, where
$E$ is the set of common Lebesgue points of the functions $\displaystyle\int_\T p(u(t,x))f(x)dx$, $p(\lambda)\in C([-M,M])$, $f(x)\in L^1(\T)$ (since the spaces $C([-M,M])$, $L^1(\T)$ are separable, this set has full measure). By the choice of $t_0$ we have
$$
J_k\doteq\int_{\T}|u(n_kt_0,y)-I|dy=\int_{\T}|u(n_kt_0,n_k x)-I|dx\mathop{\to}_{k\to\infty} 0.
$$
Further, applying contraction property (\ref{contr}) for $u_1=u$, $u_2=I$, we obtain that for a.e. $t>kt_0$
$$
\int_{\T}|u(t,x)-I|dy\le\int_{\T}|u(n_kt_0,x)-I|dy=J_k
$$
where we also use the fact that $n_kt_0\in E$.
Since $J_k\to 0$ as $k\to\infty$, the above estimate implies that
$$
\esslim_{t\to+\infty} \int_{\T}|u(t,x)-I|dy=0,
$$
and (\ref{dec}) follows.
\end{proof}

Let $[a,b]$ be the maximal segment such that $-M\le a<I<b\le M$, and that $\varphi(u)$ is affine on $[a,b]$. If such segments do not exist, i.e., $\varphi(u)$ is not affine in any vicinity of $I$, we set $a=b=I$.
Let $s_{a,b}(u)=\min(b,\max(a,u))$ be the cut-off function.

\begin{corollary}\label{cor2}
Under the above notations
$$
\esslim_{t\to+\infty}\int_\T |u(t,x)-s_{a,b}(u(t,x))|dx=0.
$$
\end{corollary}

\begin{proof}
If $b<M$ then $\varphi(u)$ is not affine in any vicinity of $b$ (otherwise, the segment $[a,b]$ can be enlarged). Let $v(t,x)$ be an e.s. of (\ref{1}), (\ref{2}) with initial function $v_0(x)=u_0(x)+b-I\ge u_0(x)$. By the comparison principle $v(t,x)\ge u(t,x)$. Further, $\displaystyle\int_{\T} v_0(x)dx=b$ and by Proposition~\ref{pro1}
$$
\esslim_{t\to+\infty}\int_{\T} |v(t,x)-b|dx=0.
$$
From this relation and the inequality
$$
\int_\T (u(t,x)-b)^+dx\le\int_\T (v(t,x)-b)^+dx\le\int_\T |v(t,x)-b|dx
$$
it follows that
\begin{equation}\label{6}
\esslim_{t\to+\infty} \int_\T (u(t,x)-b)^+dx=0.
\end{equation}
Analogously, in the case $a>-M$, $u(t,x)\ge w(t,x)$, where $w(t,x)$ is an e.s. of (\ref{1}), (\ref{2}) with initial function $w_0(x)=u_0(x)+a-I\le u_0(x)$. By Proposition~\ref{pro1} the e.s. $w(t,x)$ decays as $t\to+\infty$ to the constant $a$. This implies the limit relations
\begin{equation}\label{7}
\esslim_{t\to+\infty} \int_\T (a-u(t,x))^+dx=0.
\end{equation}
For $b=M$, $a=-M$ relations (\ref{6}), (\ref{7}) are evident.
Since $|u-s_{a,b}(u)|=(u-b)^++(a-u)^+$, the desired statement readily follows from
(\ref{6}), (\ref{7}).
\end{proof}

Observe, that in the case $a=b=I$ Corollary~\ref{cor2} reduces to Proposition~\ref{pro1}. 

\begin{lemma}\label{lem3}
Assume that $u(x),v(x)\in L^\infty(\T)$, $a\le v(x)\le b$, $\displaystyle\int_\T u(x)dx=I\in [a,b]$. Then there exists a function $w(x)\in L^\infty(\T)$ such that $a\le w(x)\le b$, $\displaystyle\int_\T w(x)dx=I$, and
$$
\|u-w\|_{L^1(\T)}\le 2\|u-v\|_{L^1(\T)}.$$
\end{lemma}

\begin{proof}
We denote $\displaystyle I'=\int_\T v(x)dx$, $\varepsilon=\|u-v\|_{L^1(\T)}$. It is clear that $|I'-I|\le\varepsilon$. If $I'=I$, there is nothing to prove, we can take $w=v$.
Assume that $I'>I$. In this case we seek the function
$w(x)$ in the form $w=sa+(1-s)v$, where the parameter $s\in (0,1)$ is defined by
the requirement $\displaystyle\int_\T w(x)dx=I$, which reduces to the equality $sa+(1-s)I'=I$.
Solving the corresponding equation, we find $s=\frac{I'-I}{I'-a}$. By the construction
\begin{eqnarray*}
\|u-w\|_1=\|u-v+s(v-a)\|_1\le \|u-v\|_1+s\|v-a\|_1=\\ \|u-v\|_1+s(I'-a)=\varepsilon+I'-I\le 2\varepsilon,
\end{eqnarray*}
where we denote $\|\cdot\|_1=\|\cdot\|_{L^1(\T)}$.

In the case $I'<I$ the proof is similar. In this case we take $w=sb+(1-s)v$, where
$s=\frac{I-I'}{b-I'}$.
\end{proof}

By the construction $\varphi(u)-cu$ is constant on $[a,b]$ for some $c\in\R$.
Assume that $v_0(x)\in L^\infty(\T)$, $a\le v_0(x)\le b$, and that $v(t,x)$ is an e.s. of (\ref{1}), (\ref{2}) with initial data $v_0(x)$. By the comparison principle
$a\le v(t,x)\le b$, and equation (\ref{1}) in such class of e.s. reduces to the form $v_t+cv_x-g(v)_{xx}=0$.

\begin{proposition}\label{pro2}
Assume that the function $g(u)$ is not constant in any vicinity of $\displaystyle I=\int_\T v_0(x)dx$. Then the decay property holds
$$
\esslim_{t\to+\infty} v(t,x)=I \ \mbox{ in } L^1(\T).
$$
\end{proposition}

\begin{proof}
Evidently, $v(t,x)=w(t,x-ct)$, where $w(t,x)$ is an e.s. to the Cauchy problem for the equation
\begin{equation}\label{8}
w_t-g(w)_{xx}=0
\end{equation}
with initial data $v_0(x)$.
The sequence $w_n=w(n^2t,nx)$, $n\in\N$, consists of $x$-periodic e.s. of the same equation. Without loss of generality we assume that some subsequence $w_{n_k}$ converges to a measure valued function $\nu_{t,x}\in \MV(\Pi)$ in the sense of (\ref{4}).
Applying Lemma~\ref{lem2} and relation (\ref{4}) (with $p(\lambda)=\lambda$), we find that identity (\ref{5}) holds. By Theorem~\ref{th1} $g(u)=\const$ on $[a(t,x),b(t,x)]=\Co\supp\nu_{t,x}$ for a.e. $(t,x)\in\Pi$. In view of (\ref{5}) $I\in (a(t,x),b(t,x))$ whenever $a(t,x)<b(t,x)$. By our assumption $g(u)$ cannot be constant on such the intervals. Thus, $a(t,x)=b(t,x)=I$ for a.e. $(t,x)\in\Pi$ and $\nu_{t,x}(\lambda)=\delta(\lambda-I)$. By Theorem~\ref{thT}
$w_{n_k}\to I$ as $k\to\infty$ in $L^1_{loc}(\R_+\times\T)$. In the same way as in the proof of Proposition~\ref{pro1}, we derive from this property that
$$
\esslim_{t\to+\infty}\int_\T |w(t,x)-I|dx=0.
$$
To complete the proof, we notice that
$$
\int_\T |v(t,x)-I|dx=\int_\T |w(t,x-ct)-I|dx=\int_\T |w(t,x)-I|dx.
$$
\end{proof}

Let $[a',b']$ be the maximal interval such that $a\le a'<I<b'\le b$ and that $g(\lambda)$ is constant on $[a',b']$. If such the intervals do not exist (that is, either $a=b=I$ or $g(\lambda)$ is not constant in any vicinity of $I$), we set $a'=b'=I$.

\begin{corollary}\label{cor3}
Under the above notations
$$
\esslim_{t\to+\infty}\int_\T |v(t,x)-s_{a',b'}(v(t,x))|dx=0.
$$
\end{corollary}

\begin{proof}
The proof is similar to the proof of Corollary~\ref{cor2}.
Assuming that $b'<b$ we can choose a constant $l\in [a,b]$ such that
$\displaystyle I(l)=\int_\T \max(l,v_0(x))dx=b'$. This follows from continuity of $I(l)$ and the relations $I(a)=I\le b'$, $I(b)=b>b'$. Let $w(t,x)$ be an e.s. of (\ref{1}), (\ref{2}) with initial function $w_0(x)=\max(l,v_0(x))\ge v_0(x)$. By the comparison principle $w(t,x)\ge v(t,x)$. Observe that $a\le w_0(x)\le b$, $\displaystyle\int_\T w_0(x)dx=b'$, and $g(u)$ is not constant in any vicinity of $b'$ (otherwise, we can enlarge the end $b'$ of the maximal interval $[a',b']$, which is impossible). By Proposition~\ref{pro2}
$$
\esslim_{t\to+\infty}\int_\T|w(t,x)-b'|dx=0.
$$
Since $(v-b')^+\le (w-b')^+\le |w-b'|$, we conclude that
\begin{equation}\label{6p}
\esslim_{t\to+\infty}\int_\T(v(t,x)-b')^+dx=0.
\end{equation}
Obviously, (\ref{6p}) is satisfied also for $b'=b$.
Analogously, we prove that
\begin{equation}\label{7p}
\esslim_{t\to+\infty}\int_\T(a'-v(t,x))^+dx=0.
\end{equation}
For that we compare $v(t,x)$ and e.s. $w(t,x)$ of (\ref{1}), (\ref{2}) with initial function $w_0(x)=\min(l,v_0(x))\le v_0(x)$, where $l\in [a,b]$ is chosen to satisfy
the equality $\displaystyle\int_\T w_0(x)dx=a'$.

The desired statement readily follows from relations (\ref{6p}), (\ref{7p}). The proof is complete.
\end{proof}

Notice that the segment $[a',b']$ is the maximal segment such that $-M\le a'<I<b'\le M$ and that the functions $\varphi(u)-cu$, $g(u)$ are constant on $[a',b']$ for some $c\in\R$ (and $a'=b'=I$ if such segments do not exist).

Let $E\subset\R_+$ be the set of full measure defined in the proof of Proposition~\ref{pro1} above. Recall that this set consists of common Lebesgue points of the functions 
$$t\to\int_\T p(u(t,x))f(x)dx, \quad p(\lambda)\in C([-M,M]), \ f(x)\in L^1(\T).$$

\begin{lemma}\label{addlem}
Assume that $u(t,x)$ is an e.s. of (\ref{1}), (\ref{2}) with a periodic initial function and $t_0\in E$. Then
$$
u(t,x)\to u(t_0,x) \ \mbox{ in } L^1(\T)
$$
as $t\to t_0+$, $t\in E$.
\end{lemma}

\begin{proof}
If $t,t_0\in E$ then these points are Lebesgue points of the functions $\displaystyle\int_\T |u(t,x)-k|dx$ for all $k\in\R$, and it follows from (\ref{contr}) with $u_1=u(t,x)$, $u_2\equiv k$ that for each $E\ni t>t_0$
\begin{equation}\label{al1}
\int_\T |u(t,x)-k|dx\le\int_\T |u(t_0,x)-k|dx.
\end{equation}
Let $M=\|u\|_\infty$. Integrating (\ref{al1}) with respect to $k\in [-M,M]$ and taking into account that for $|\lambda|\le M$
$$
\int_{-M}^M |\lambda-k|dk=\lambda^2+M^2,
$$
we obtain that for all $t\in E$, $t>t_0$
\begin{equation}\label{al2}
\int_\T (u(t,x))^2 dx\le\int_\T (u(t_0,x))^2 dx.
\end{equation}
Further, applying (\ref{1}) to a test function $f(x)\in C_0^\infty(\R)$, we obtain that
$$
\frac{d}{dt}\int_\R u(t,x)f(x)dx=\int_\R [\varphi(u(t,x))f'(x)+g(u(t,x))f''(x)]dx.
$$
Since the right-hand side of this equality is bounded, we see that the function $\displaystyle\int_\R u(t,x)f(x)dx$ is Lipshitz continuous with respect to $t$. In particular, this function is continuous on $E$. This implies that for every $f(x)\in C_0^\infty(\R)$
$$
\lim_{E\ni t\to t_0}\int_\R u(t,x)f(x)dx=\int_\R u(t_0,x)f(x)dx.
$$
Notice that the family $u(t,\cdot)$, $t\in E$, is bounded in $L^\infty(\R)$ and it follows from the above relation that $u(t,\cdot)\rightharpoonup u(t_0,\cdot)$ weakly-$*$ in $L^\infty(\R)$ (and also in $L^\infty(\T)$, by the periodicity)  as $E\ni t\to t_0$. In particular,
\begin{equation}\label{al3}
\lim_{E\ni t\to t_0}\int_\T u(t,x)u(t_0,x)dx=\int_\T (u(t_0,x))^2dx.
\end{equation}
By (\ref{al2}), (\ref{al3})
\begin{eqnarray*}
\limsup_{E\ni t\to t_0+}\int_\T (u(t,x)-u(t_0,x))^2dx=\limsup_{E\ni t\to t_0+}\int_\T \{(u(t,x))^2dx-2u(t,x)u(t_0,x)+(u(t_0,x))^2\}dx\le\\
2\int_\T (u(t_0,x))^2dx-2\lim_{E\ni t\to t_0}\int_\T u(t,x)u(t_0,x)dx=0.
\end{eqnarray*}
Hence, $u(t,x)\to u(t_0,x)$ in $L^2(\T)$ as $E\ni t\to t_0+$. In view of continuity of the embedding $L^2(\T)\subset L^1(\T)$, this completes the proof.
\end{proof}

From Propositions~\ref{pro1},~\ref{pro2} we derive the following result.

\begin{theorem}\label{th2}
$$
\esslim_{t\to+\infty}\int_\T |u(t,x)-s_{a',b'}(u(t,x))|dx=0.
$$
\end{theorem}

\begin{proof}
Let $E$ be the set of full measure defined above in the proof of Proposition~\ref{pro1}. As easily follows from Corollary~\ref{cor2},
$$\int_\T |u(t,x)-s_{a,b}(u(t,x))|dx\to 0$$ as $t\to+\infty$, $t\in E$. Let $t_k\in E$, $k\in\N$, be a strictly increasing sequence such that $t_k\to+\infty$ as $k\to\infty$. Then $$\varepsilon_k\doteq\int_\T |u(t_k,x)-s_{a,b}(u(t_k,x))|dx\to 0 \ \mbox{ as } k\to\infty.$$
By Lemma~\ref{lem3} we can find functions $v_{0k}\in L^\infty(\T)$ with the properties $a\le v_{0k}\le b$, $\displaystyle\int v_{0k}(x)dx=I$, and $\|u(t_k,\cdot)-v_{0k}\|_{L^1(\T)}\le 2\varepsilon_k$. Let $u=v_k=v_k(t,x)$ be an e.s. of (\ref{1}), (\ref{2}) in the half-plane $t>t_k$ with initial data $u(t_k,x)=v_{0k}(x)$. As easily follows from the requirement $t_k\in E$ and Lemma~\ref{addlem}, $u=u(t,x)$ is an e.s. of (\ref{1}), (\ref{2}) in the same half-plane $t>t_k$ with initial data
$u(t_k,x)$. By the $L^1(\T)$-contraction property, for a.e. $t>t_k$
\begin{equation}\label{9}
\int_\T |u(t,x)-v_k(t,x)|dx\le\int_\T |u(t_k,x)-v_{0k}(x)|dx\le 2\varepsilon_k.
\end{equation}
On the other hand, by Corollary~\ref{cor3}
\begin{equation}\label{10}
\esslim_{t\to+\infty}\int_{\T} |v_k(t,x)-s_{a',b'}(v_k(t,x))|dx=0.
\end{equation}
Since
\begin{eqnarray*}
|u-s_{a',b'}(u)|\le |u-v_k|+|v_k-s_{a',b'}(v_k)|+|s_{a',b'}(v_k)-s_{a',b'}(u)|\le \\
|u-v_k|+|v_k-s_{a',b'}(v_k)|+|v_k-u|=2|u-v_k|+|v_k-s_{a',b'}(v_k)|,
\end{eqnarray*}
it follows from (\ref{9}), (\ref{10}) that for all $k\in\N$
$$
\esslimsup_{t\to+\infty}\int_\T |u(t,x)-s_{a',b'}(u(t,x))|dx\le 4\varepsilon_k.
$$
Passing to the limit as $k\to\infty$, we obtain
$$
\esslimsup_{t\to+\infty}\int_\T |u(t,x)-s_{a',b'}(u(t,x))|dx=0,
$$
as was to be proved.
\end{proof}

Now we are ready to prove our main result.

\begin{proof}[Proof of Theorem~\ref{thM}]
Let $t_k\in E$ be the same sequence as in the proof of previous Theorem~\ref{th2}.
We set $v_{0k}(x)=s_{a',b'}(u(t_k,x))$. Then
$$
\varepsilon_k\doteq\|u(t_k,\cdot)-v_{0k}\|_{L^1(\T)}\to 0 \ \mbox{ as } k\to\infty.
$$
Observe that $\varphi(u)-cu=\const$, $g(u)=\const$ for $u\in [a',b']$ and (\ref{1}) reduces to the simple equation $u_t+cu_x=0$. Therefore, the unique e.s. of the Cauchy problem for (\ref{1}) in the half-plane $t>t_k$ satisfying the condition $u(t_k,\cdot)=v_{0k}$ has the form $u=v_k(x-ct)$, where $v_k=v_{0k}(x+ct_k)$.
Since $u=u(t,x)$ is an e.s. of the same problem for equation (\ref{1}) with the Cauchy data $u(t_k,x)$, then by the contraction property (\ref{contr}) for a.e. $t>t_k$
\begin{equation}\label{11}
\int_\T |u(t,x)-v_k(x-ct)|dx\le\int_\T |u(t_k,x)-v_{0k}(x)|dx\le\varepsilon_k.
\end{equation}
As follows from (\ref{11}) and the translation invariance of the measure $dx$ on $\T$, for all $l,k\in\N$, $l>k$
\begin{eqnarray*}
\int_\T |v_l(x)-v_k(x)|dx=\int_\T |v_l(x-ct)-v_k(x-ct)|dx\le \int_\T |u(t,x)-v_k(x-ct)|dx+ \\ \int_\T |u(t,x)-v_l(x-ct)|dx\le \varepsilon_k+\varepsilon_l\le 2\sup_{l\ge k}\varepsilon_l\mathop{\to}_{k\to\infty} 0.
\end{eqnarray*}
This means that $v_k$, $k\in\N$, is a Cauchy sequence in $L^1(\T)$. By the completeness of $L^1(\T)$ this sequence converges to a function $v(y)$ in $L^1(\T)$. Since $a'\le v_k\le b'$, then $a'\le v\le b'$ as well. In particular, the functions $\varphi(u)-cu$, $g(u)$ are constant on $[\alpha(v),\beta(v)]$. Further, in view of (\ref{11}) for a.e. $t>t_k$
\begin{eqnarray*}
\int_\T |u(t,x)-v(x-ct)|dx\le\int_\T |u(t,x)-v_k(x-ct)|dx+\int_\T|v_k(y)-v(y)|dy\le \\ \varepsilon_k+\|v_k-v\|_{L^1(\T)}\mathop{\to}_{k\to\infty} 0.
\end{eqnarray*}
Evidently, this implies the desired asymptotic property (\ref{ass}). From this property it follows that
$$
\int_\T v(y)dy=\int_\T v(x-ct)dx=\esslim_{t\to+\infty}\int_\T u(t,x)dx=I.
$$
The proof is complete.
\end{proof}

Let us introduce the nonlinear operator $T$ on $L^\infty(\T)$, which associates an initial function $u_0$ with the profile $v(y)=T(u_0)(y)$ of the limit traveling wave defined by (\ref{ass}).

In the conclusion, we prove that the operator $T$ does not increase the $L^1$-distance (for conservation laws this result is proved in \cite[Theorem 4.1]{PaMZM}).

\begin{theorem}\label{th3}
Let $u_{01},u_{02}\in L^\infty(\T)$, and $v_1=T(u_{01})(y)$, $v_2=T(u_{02})(y)$. Then
\begin{equation}\label{12}
\int_\T |v_1(y)-v_2(y)|dx\le\int_\T |u_{01}(x)-u_{02}(x)|dx.
\end{equation}
\end{theorem}

\begin{proof}
Let $u_1(t,x)$, $u_2(t,x)$ be e.s. of (\ref{1}), (\ref{2}) with initial functions
$u_{01}(x)$, $u_{02}(x)$, respectively. By (\ref{ass})
$$
\esslim_{t\to+\infty}
(u_1(t,x)-v_1(x-c_1t))=\esslim_{t\to+\infty} (u_2(t,x)-v_2(x-c_2t))=0 \ \mbox{ in } L^1(\T).
$$
If $c_1=c_2=c$, then with the help of (\ref{contr}) we find
\begin{eqnarray*}
\int_\T |v_1(y)-v_2(y)|dy=\int_\T |v_1(x-ct)-v_2(x-ct)|dy\le \\ \int_\T |u_1(t,x)-u_2(t,x)|dx+\int_\T|u_1(t,x)-v_1(x-ct)|dx+\int_\T|u_2(t,x)-v_2(x-ct)|dx \\ \le \int_\T |u_{01}(x)-u_{02}(x)|dx+\int_\T|u_1(t,x)-v_1(x-ct)|dx+\int_\T|u_2(t,x)-v_2(x-ct)|dx.
\end{eqnarray*}
In the essential limit as $t\to+\infty$ this implies (\ref{12}).

In the case when $c_1\not=c_2$ Theorem~\ref{thM} implies that the intervals $(\alpha(v_1),\beta(v_1))$, $(\alpha(v_2),\beta(v_2))$ cannot intersect. Therefore,
\begin{eqnarray*}
\int_\T |v_1(y)-v_2(y)|dy=\left|\int_\T v_1(y)dy-\int_T v_2(y)dy\right|=\\
\left|\int_\T u_{01}(x)dx-\int_\T u_{02}(x)dx\right|\le \int_\T |u_{01}(x)-u_{02}(x)|dx,
\end{eqnarray*}
and (\ref{12}) follows.
\end{proof}

\section{Acknowledgements}
This work was supported by the Ministry of Education and Science of the Russian  Federation (project no. 1.445.2016/1.4) and by the Russian Foundation for Basic Research (grant 18-01-00258-a.)

\end{document}